\documentclass[11pt]{article}
\usepackage[a4paper,margin=28mm]{geometry}
\usepackage{amsmath,amssymb,amsthm,aliascnt}
\usepackage{xcolor}
\usepackage{hyperref}
\hypersetup{colorlinks=true,linkcolor=cyan,citecolor=red,urlcolor=red,filecolor=red}
\numberwithin{equation}{section}

\newtheorem{theorem}{Theorem}[section]
\newaliascnt{conjecture}{theorem}
\newtheorem{conjecture}[conjecture]{Conjecture}
\aliascntresetthe{conjecture}
\newaliascnt{lemma}{theorem}
\newtheorem{lemma}[lemma]{Lemma}
\aliascntresetthe{lemma}
\newaliascnt{proposition}{theorem}
\newtheorem{proposition}[proposition]{Proposition}
\aliascntresetthe{proposition}

\DeclareMathOperator{\GF}{GF}
\DeclareMathOperator{\supp}{supp}
\DeclareMathOperator{\wt}{wt}

\begin{document}

\begin{center}
{\Large\bfseries
A counterexample to the circuit-cocircuit intersection conjecture
}
\end{center}

\begin{center}
Houshan Fu\par

School of Mathematics and Information Science, Guangzhou University\\
Guangzhou 510006, Guangdong, P. R. China\par

Email address: fuhoushan@gzhu.edu.cn
\end{center}

\begin{abstract}
Oxley conjectured that a matroid with a circuit-cocircuit intersection of
size $k\geq 4$ has one of size $k-2$.
Using the standard polynomial representation of the classical extended
binary Golay code, we construct an identically self-dual binary matroid of
rank $12$ on $24$ elements. Its circuit-cocircuit intersection sizes are
exactly $0,2,4,6,8$, and $12$, disproving the conjecture for $k=12$.
The proof is self-contained and does not rely on the full weight distribution
of the code or on an enumeration of all circuits.
\end{abstract}

\section{Introduction}
All matroids considered are finite, and the matroid terminology follows
Oxley~\cite{Oxley2011}. For a matroid $M$, write $E(M)$, $r(M)$, and $M^*$
for its ground set, rank, and dual, respectively. Denote its families of
circuits and cocircuits by $\mathcal C(M)$ and
$\mathcal C^*(M)=\mathcal C(M^*)$, respectively.
A matroid $M$ is \emph{identically self-dual} if $M=M^*$.
A matroid is \emph{binary} if it is representable over the field
$\GF(2)$ with two elements.

A \emph{circuit-cocircuit intersection} is the intersection of a circuit
and a cocircuit of the same matroid. No circuit-cocircuit intersection has
size one~\cite[Proposition~2.1.11]{Oxley2011}. A matroid is binary if and
only if every circuit-cocircuit intersection has even size
\cite[Theorem~9.1.2]{Oxley2011}.
Oxley~\cite[Theorem~2.7]{Oxley1984} proved that a matroid with a circuit-cocircuit
intersection of size at least four has one of size four.
The following conjecture asks whether the intersection size can always
be reduced by two.

\begin{conjecture}[{\cite[Conjecture~15.9.5]{Oxley2011}}]
\label{conj:oxley}
Let $M$ be a matroid with a $k$-element circuit-cocircuit intersection,
where $k\geq4$.
Then $M$ has a $(k-2)$-element circuit-cocircuit intersection.
\end{conjecture}

Recently, Shin proved the cases $k\leq6$ for arbitrary matroids
in \cite[Theorem~2.2]{Shin2021} and settled the case $k=7$
in \cite[Theorem~2.7]{Shin2023}.
Earlier, Kingan~\cite[Theorem~4.2]{Kingan1999} established
\autoref{conj:oxley} for graphic matroids.
Kingan and Lemos~\cite[Corollary~1.2]{KinganLemos2006} subsequently proved
a stronger result for regular matroids: every element of a $k$-element
circuit-cocircuit intersection, where $k\geq4$, belongs to one of size $k-2$.

The initial motivation for this work was to determine whether
\autoref{conj:oxley} holds for the broader class of binary matroids.
Within this class, Kingan~\cite[Theorem~2.3]{Kingan1999} proved that a matroid
with a circuit-cocircuit intersection of size at least six has one of size six.
Kingan and Lemos~\cite[pp.~472--473]{KinganLemos2006} used self-dual
binary codes to construct matroids in their study of this conjecture.
Their paper also contains a rank-nine example due to Oxley with circuit-cocircuit
intersections of sizes ten and eight~\cite[pp.~473, 480]{KinganLemos2006}.

The matroid represented by a generator matrix of the classical extended
binary Golay code disproves the conjecture within the class of binary
matroids: it has a twelve-element circuit-cocircuit intersection but none
of size ten.

\begin{theorem}\label{thm:main}
There is an identically self-dual binary matroid $M$ of rank $12$ on $24$
elements such that
\[
 \bigl\{|C\cap C^*|:
 C\in\mathcal C(M),\ C^*\in\mathcal C^*(M)\bigr\}
 =\{0,2,4,6,8,12\}.
\]
In particular, $M$ has a $12$-element circuit-cocircuit intersection and
no $10$-element circuit-cocircuit intersection.
\end{theorem}

The proof uses the code's duality and weight properties to characterize
the circuits and cocircuits as supports of vectors of weight eight or twelve.
If two such supports meet in ten elements, both have size twelve and their
symmetric difference is the support of a vector of weight four. Such a
vector contradicts the minimum weight eight. The required code properties
are proved directly from the standard polynomial representation.

\section{Construction and weight bounds}

Let $E$ be the set of column labels of a matrix $A$ over a field $\mathbb{F}$.
For each $e\in E$, write $v_e$ for the column labelled by $e$.
Set
\[
 \mathcal I=\{I\subseteq E:(v_e)_{e\in I}
 \text{ is linearly independent over }\mathbb{F}\}.
\]
Then $M[A]=(E,\mathcal I)$ is a matroid, called the \emph{vector matroid}
of $A$~\cite[Proposition~1.1.1]{Oxley2011}.
The \emph{row space} $R(A)$ is the subspace spanned over $\mathbb{F}$ by the rows of $A$.
For each positive integer $n$, write $V(n,2)=\GF(2)^n$.
The \emph{symmetric difference} of sets $X$ and $Y$ is
$X\mathbin{\triangle}Y=(X-Y)\cup(Y-X)$.

For a vector $u=(u_e)_{e\in E}$ over $\GF(2)$, its \emph{support} is
$\supp(u)=\{e\in E:u_e\ne0\}$, and its \emph{weight} is
$\wt(u)=|\supp(u)|$.
The \emph{incidence vector} of $X\subseteq E$ is the vector with
$e$-coordinate one for $e\in X$ and zero for $e\in E-X$.
Orthogonality is taken with respect to the usual dot product over $\GF(2)$.

The construction follows the polynomial representation of the extended
binary Golay code in Huffman and Pless~\cite[Section~10.1.2]{HuffmanPless2003}.
The binary Golay code of length $23$ has generator polynomial
\begin{equation}\label{eq:g}
 g(x)=1+x+x^5+x^6+x^7+x^9+x^{11}\in\GF(2)[x].
\end{equation}
The reciprocal polynomial of $g$ is
\[
 h(x)=x^{11}g(x^{-1})=1+x^2+x^4+x^5+x^6+x^{10}+x^{11}.
\]
Direct multiplication in $\GF(2)[x]$ gives
\begin{equation}\label{eq:product}
 g(x)h(x)=1+x+x^2+\cdots+x^{22}.
\end{equation}
For $0\leq i\leq11$ and $0\leq j\leq22$, define $a_{ij}\in\GF(2)$ by
\begin{equation}\label{eq:row-coefficients}
 x^i g(x)=\sum_{j=0}^{22}a_{ij}x^j.
\end{equation}
For $0\leq i\leq11$, put
$u_i=(a_{i0},a_{i1},\ldots,a_{i22},1)\in V(24,2)$.
Let $A$ be the $12\times24$ matrix over $\GF(2)$ with rows
$u_0,u_1,\ldots,u_{11}$ and columns labelled $0,1,\ldots,23$ in this order.
Set $E=\{0,1,\ldots,23\}$ and $M=M[A]$.
The matrix $A$ and the matroid $M$ are fixed for the remainder of the paper.
For $0\leq i\leq11$, the polynomial $x^i g(x)$ has seven nonzero
coefficients. The appended entry one is the overall parity-check
coordinate, and $\wt(u_i)=8$.
The resulting row space $R(A)$ is the extended binary Golay code,
with generator matrix $A$ and associated matroid $M[A]$.
The \emph{orthogonal subspace} of $R(A)$ is
\[
 R(A)^\perp=\{v\in V(24,2):u\cdot v=0\text{ for every }u\in R(A)\}.
\]

The next lemma uses the dot products of the rows to control the weights
of their sums.

\begin{lemma}\label{lem:orthogonal}
The row space $R(A)$ has dimension $12$ and satisfies
\[
 R(A)=R(A)^\perp.
\]
Moreover, every vector in $R(A)$ has weight divisible by four.
\end{lemma}

\begin{proof}
We first show that $u_0,\ldots,u_{11}$ form a basis of $R(A)$.
They span $R(A)$ by definition.
Suppose $\sum_{i=0}^{11}c_i u_i=0$, where $c_i\in\GF(2)$.
In the first $23$ coordinates, this relation reads
$\sum_{i=0}^{11}c_i a_{ij}=0$ for $0\leq j\leq22$.
Multiplication by $x^j$ and summation over $0\leq j\leq22$, together with
\eqref{eq:row-coefficients}, transform these coordinate relations into
\[
 0=\sum_{j=0}^{22}\Bigl(\sum_{i=0}^{11}c_i a_{ij}\Bigr)x^j
  =\sum_{i=0}^{11}c_i x^i g(x)
  =g(x)\sum_{i=0}^{11}c_i x^i.
\]
Since $g(x)\ne0$ and $\GF(2)[x]$ is an integral domain,
$\sum_{i=0}^{11}c_i x^i=0$. Comparing coefficients shows that
$c_i=0$ for every $i$.
Thus $u_0,\ldots,u_{11}$ form a basis of $R(A)$, and $\dim R(A)=12$.

To prove orthogonality of the rows, substitute $h(x)=x^{11}g(x^{-1})$ into
\eqref{eq:product} and divide by $x^{11}$. The resulting identity in the
Laurent polynomial ring $\GF(2)[x,x^{-1}]$ is
\[
 g(x)g(x^{-1})=\sum_{d=-11}^{11}x^d.
\]
For $0\leq i,j\leq11$, multiply \eqref{eq:row-coefficients} for $i$
by the same identity for $j$ with $x$ replaced by $x^{-1}$:
\[
 \sum_{k,l=0}^{22}a_{ik}a_{jl}x^{k-l}
 =x^{i-j}g(x)g(x^{-1})
 =\sum_{d=-11}^{11}x^{d+i-j}.
\]
On the left, the constant terms are precisely those with $k=l$.
On the right, only $d=j-i$ contributes a constant term, and
$-11\leq j-i\leq11$.
Hence $\sum_{k=0}^{22}a_{ik}a_{jk}=1$. Since the last coordinate of each row is one,
$u_i\cdot u_j=1+1=0$.
By bilinearity, any two linear combinations of the rows are orthogonal.
Thus
\begin{equation}\label{eq:self-orthogonal}
 R(A)\subseteq R(A)^\perp.
\end{equation}
Since $\dim R(A)^\perp=24-\dim R(A)=12$, the inclusion
\eqref{eq:self-orthogonal} is an equality.

To prove the divisibility assertion, consider $u,v\in V(24,2)$.
Binary addition corresponds to symmetric difference of supports:
\begin{equation}\label{eq:binary-support}
 \supp(u+v)=\supp(u)\mathbin{\triangle}\supp(v).
\end{equation}
Consequently, the weights satisfy
\begin{equation}\label{eq:weight}
 \wt(u+v)=\wt(u)+\wt(v)-2|\supp(u)\cap\supp(v)|.
\end{equation}
For $u,v\in R(A)$, \eqref{eq:self-orthogonal} ensures $u\cdot v=0$.
In $\GF(2)$,
\[
 0=u\cdot v=\sum_{e\in E}u_ev_e
 =\sum_{e\in\supp(u)\cap\supp(v)}1.
\]
Thus $|\supp(u)\cap\supp(v)|$ is even, and \eqref{eq:weight} implies
\begin{equation}\label{eq:weight-congruence}
 \wt(u+v)\equiv\wt(u)+\wt(v)\pmod4.
\end{equation}
For an arbitrary $u\in R(A)$, the basis coefficients lie in $\GF(2)$,
so $u=\sum_{i\in I}u_i$ for some $I\subseteq\{0,\ldots,11\}$.
Every partial sum belongs to $R(A)$. Repeated application of
\eqref{eq:weight-congruence}, using $\wt(u_i)=8$, gives
\[
 \wt(u)\equiv\sum_{i\in I}\wt(u_i)=8|I|\equiv0\pmod4.
\]
The empty sum corresponds to the zero vector and has weight zero.
\end{proof}

The divisibility by four reduces the minimum-weight bound to the exclusion
of weight four. For this step, the proof uses the consecutive-root argument
for cyclic codes described by Huffman and
Pless~\cite[Theorem~4.5.3 and Example~4.5.5]{HuffmanPless2003}, including its
application to the binary Golay code.

\begin{lemma}\label{lem:weight}
Every nonzero vector in $R(A)$ has weight at least eight.
\end{lemma}

\begin{proof}
By \autoref{lem:orthogonal}, a nonzero vector in $R(A)$ can have
weight less than eight only if its weight is four.
Suppose, for a contradiction, that $u\in R(A)$ has weight four.
Since $u_0,u_1,\ldots,u_{11}$ form a basis of $R(A)$, write
$u=\sum_{i=0}^{11}c_i u_i$ with $c_i\in\GF(2)$.
Let $f(x)$ be the polynomial whose coefficients are the first $23$
coordinates of $u$.
For $0\leq j\leq22$, the coefficient of $x^j$ in $f$ is
$\sum_{i=0}^{11}c_i a_{ij}$. Substitution from \eqref{eq:row-coefficients}
expresses $f$ as a multiple of $g$:
\begin{equation}\label{eq:multiple}
 f(x)=\sum_{j=0}^{22}\Bigl(\sum_{i=0}^{11}c_i a_{ij}\Bigr)x^j=\sum_{i=0}^{11}c_i\Bigl(\sum_{j=0}^{22}a_{ij}x^j\Bigr)=\Bigl(\sum_{i=0}^{11}c_i x^i\Bigr)g(x).
\end{equation}
Deleting the last coordinate of $u$ removes either zero or one nonzero
entry. Since $\wt(u)=4$, the polynomial $f$ has $t\in\{3,4\}$ nonzero
coefficients. Denote their positions by $0\leq e_1<\cdots<e_t\leq22$.
Since each nonzero coefficient equals one in $\GF(2)$, the expansion is
\begin{equation}\label{eq:sparse-expansion}
 f(x)=\sum_{l=1}^{t}x^{e_l}.
\end{equation}

Since $g$ divides $f$ by \eqref{eq:multiple}, the roots of $g$ impose
linear relations on the $t$ nonzero coefficients of $f$. The contradiction
follows once $t$ of these relations form a nonsingular homogeneous system.

Let $\alpha$ be a root of $g$ in an algebraic closure of $\GF(2)$.
Multiplying \eqref{eq:product} by $x-1$ gives $(x-1)g(x)h(x)=x^{23}-1$.
At $x=\alpha$, the left-hand side vanishes, so $\alpha^{23}=1$.
The equality $g(1)=1$ excludes $\alpha=1$. As $23$ is prime,
the multiplicative order of $\alpha$ is $23$.
In characteristic two, $(a+b)^{2^j}=a^{2^j}+b^{2^j}$ for all elements
$a,b$ of the algebraic closure and every integer $j\geq0$.
Since each coefficient of $g$ is zero or one, this identity implies
$g(\alpha^{2^j})=g(\alpha)^{2^j}=0$ for every integer $j\geq0$.
In particular, $\alpha,\alpha^2$, and $\alpha^4$ are roots of $g$.
Since $2^8\equiv3\pmod{23}$ and $\alpha^{23}=1$,
$\alpha^{2^8}=\alpha^3$.
Thus, $\alpha,\alpha^2,\alpha^3,\alpha^4$ are roots of both $g$ and $f$.
For $1\leq j\leq t\leq4$, evaluation of \eqref{eq:sparse-expansion}
at $\alpha^j$ gives
\[
 0=f(\alpha^j)=\sum_{l=1}^{t}\alpha^{je_l}
 \qquad (1\leq j\leq t).
\]
Equivalently,
\begin{equation}\label{eq:root-equations}
 \bigl(\alpha^{je_l}\bigr)_{1\leq j,l\leq t}
 \begin{pmatrix}1\\\vdots\\1\end{pmatrix}
 =\begin{pmatrix}0\\\vdots\\0\end{pmatrix}.
\end{equation}
Since $\alpha^{23}=1$, both $\alpha\ne0$ and
$\alpha^{e_l}\ne0$ for $1\leq l\leq t$.
For $l,m\in\{1,\ldots,t\}$, if $\alpha^{e_l}=\alpha^{e_m}$, then
$23\mid(e_l-e_m)$ because $\alpha$ has multiplicative order $23$.
The bound $|e_l-e_m|<23$ forces $e_l=e_m$, and the strict ordering of
the exponents implies $l=m$. Thus, $\alpha^{e_1},\ldots,\alpha^{e_t}$
are pairwise distinct.
After $\alpha^{e_l}$ is factored from each column $l$, the remaining
matrix is a Vandermonde matrix. The original determinant factors as
\[
 \det\bigl(\alpha^{je_l}\bigr)_{1\leq j,l\leq t}=\Bigl(\prod_{l=1}^{t}\alpha^{e_l}\Bigr)\det\bigl(\alpha^{(j-1)e_l}\bigr)_{1\leq j,l\leq t}=\Bigl(\prod_{l=1}^{t}\alpha^{e_l}\Bigr)
   \prod_{1\leq l<m\leq t}(\alpha^{e_m}-\alpha^{e_l})\ne0.
\]
The coefficient matrix is nonsingular. Left multiplication by its inverse
reduces \eqref{eq:root-equations} to $(1,\ldots,1)^T=0$,
a contradiction since $t\geq1$.
\end{proof}

\section{Circuit-cocircuit intersections}

The \emph{circuit space} and \emph{cocircuit space} of $M$ are the subspaces
of $V(24,2)$ spanned over $\GF(2)$ by the
incidence vectors of its circuits and cocircuits, respectively.
For the binary representation $A$, these spaces are $R(A)^\perp$ and
$R(A)$, respectively~\cite[Proposition~9.2.2]{Oxley2011}.
The cocircuits are the inclusion-minimal nonempty supports of vectors in
$R(A)$~\cite[Proposition~9.2.4]{Oxley2011}. By duality, the circuits are
the inclusion-minimal nonempty supports of vectors in $R(A)^\perp$.
The rank bound on circuit size, together with the weight constraints in
\autoref{lem:orthogonal} and \autoref{lem:weight}, determines these minimal supports.

\begin{proposition}\label{prop:circuits}
The matroid $M$ is identically self-dual, has rank $12$, and satisfies
\[
 \mathcal C(M)=\mathcal C^*(M)
 =\{\supp(u):u\in R(A),\ \wt(u)\in\{8,12\}\}.
\]
\end{proposition}

\begin{proof}
By \autoref{lem:orthogonal}, the circuit and cocircuit spaces of $M$
both equal $R(A)$. Equality of their inclusion-minimal nonempty supports
implies $\mathcal C(M)=\mathcal C^*(M)=\mathcal C(M^*)$.
Since the circuits determine a matroid, $M=M^*$.
Moreover, $r(M)=\operatorname{rank}A=\dim R(A)=12$.

For a circuit $C$ and $e\in C$, the set $C-\{e\}$ is independent.
Hence $|C|\leq r(M)+1=13$.
Since the incidence vector of $C$ belongs to $R(A)$,
\autoref{lem:orthogonal} and \autoref{lem:weight} imply
$8\leq|C|\leq13$ and $4\mid|C|$. These constraints restrict $|C|$ to $8$ or $12$.

Conversely, let $u\in R(A)$ have weight eight or twelve.
It remains to show that $\supp(u)$ is inclusion-minimal among the nonempty
supports in $R(A)$. Suppose that $0\ne v\in R(A)$ has
$\supp(v)\subsetneq\supp(u)$.
By \eqref{eq:binary-support},
$\supp(u+v)=\supp(u)-\supp(v)$.
The inclusion is proper, so $u+v\ne0$.
By \autoref{lem:weight}, the nonzero vectors $v$ and $u+v$ satisfy
$\wt(v)\geq8$ and $\wt(u+v)\geq8$. Since their supports partition
$\supp(u)$,
\[
 \wt(u)=\wt(v)+\wt(u+v)\geq16,
\]
contrary to $\wt(u)\in\{8,12\}$. Thus $\supp(u)$ is both a circuit and
a cocircuit.
\end{proof}

The characterization in \autoref{prop:circuits} reduces the intersection
problem to pairs of supports of vectors of weight eight or twelve. The proof first
excludes size ten and then exhibits pairs for all remaining possibilities.
\begin{proof}[Proof of \autoref{thm:main}]
Let $C\in\mathcal C(M)$ and $C^*\in\mathcal C^*(M)$, and let $u,v$ be
their incidence vectors, respectively. Write $s=|C\cap C^*|$.
By \autoref{prop:circuits}, $u,v\in R(A)$ and $|C|,|C^*|\in\{8,12\}$.
By \eqref{eq:self-orthogonal}, these incidence vectors satisfy
$0=u\cdot v=\sum_{e\in C\cap C^*}1$ in $\GF(2)$.
Thus, $s$ is even. Together with the bound
$s\leq\min\{|C|,|C^*|\}\leq12$, this parity condition implies
$s\in\{0,2,4,6,8,10,12\}$.

If $s=10$, then $|C|=|C^*|=12$.
The vector $u+v$ lies in $R(A)$ and has support $C\mathbin{\triangle}C^*$
by \eqref{eq:binary-support}. Therefore,
\[
 \wt(u+v)=|C\mathbin{\triangle}C^*|=|C|+|C^*|-2s=4.
\]
The nonzero vector $u+v$ contradicts the weight bound in
\autoref{lem:weight}, excluding $s=10$.

To realize the six remaining sizes, put $C_i=\supp(u_i)$ for
$i\in\{0,1,3\}$. From \eqref{eq:g} and \eqref{eq:row-coefficients},
\[
 C_i=\{i,i+1,i+5,i+6,i+7,i+9,i+11,23\}
 \qquad(i\in\{0,1,3\}).
\]
Since each $u_i\in R(A)$ has weight eight, \autoref{prop:circuits}
identifies $C_0,C_1,C_3$ as both circuits and cocircuits.
Inspection of these sets gives
\[
 C_0\cap C_3=\{9,23\},
 \qquad C_0\cap C_1=\{1,6,7,23\}.
\]
Set $C'=\supp(u_0+u_3)=C_0\mathbin{\triangle}C_3$ by
\eqref{eq:binary-support}. The values $|C_0|=|C_3|=8$ and
$|C_0\cap C_3|=2$ determine
\[
 |C'|=|C_0|+|C_3|-2|C_0\cap C_3|=12,
 \qquad |C_0\cap C'|=|C_0-C_3|=8-2=6.
\]
Since $u_0+u_3\in R(A)$ has weight $|C'|=12$, the set $C'$ is also both
a circuit and a cocircuit by \autoref{prop:circuits}.
Together with $|C_0\cap C_0|=8$
and $|C'\cap C'|=12$, the intersections above realize $s=2,4,6,8,12$.

To realize $s=0$, we use the complementary sets $C'$ and $E-C'$.
It remains to show that $E-C'$ is a cocircuit. Consider the all-one vector
$\mathbf1=(1,\ldots,1)\in V(24,2)$.
The even weights in $R(A)$ imply
$\mathbf1\cdot u=\sum_{e\in E}u_e=0$ in $\GF(2)$ for every $u\in R(A)$.
By \autoref{lem:orthogonal}, $\mathbf1\in R(A)^\perp=R(A)$.
By \eqref{eq:binary-support}, the vector $\mathbf1+u_0+u_3\in R(A)$
has support
\[
 \supp(\mathbf1+u_0+u_3)=E\mathbin{\triangle}C'=E-C'.
\]
Since $|E-C'|=24-|C'|=12$, \autoref{prop:circuits} identifies $E-C'$
as a cocircuit. The disjoint circuit-cocircuit pair $C',E-C'$ realizes $s=0$.
\end{proof}

\section*{Acknowledgements}
This work is supported by the Guangdong Basic and Applied Basic Research Foundation (Grant No. 2026A1515012237).


\begin{thebibliography}{9}

\bibitem{HuffmanPless2003}
W.~C.~Huffman and V.~Pless,
\emph{Fundamentals of Error-Correcting Codes},
Cambridge University Press, Cambridge, 2003.

\bibitem{Kingan1999}
S.~R.~Kingan,
Intersections of circuits and cocircuits in binary matroids,
\emph{Discrete Math.} \textbf{195} (1999), no.~1--3, 157--165.

\bibitem{KinganLemos2006}
S.~R.~Kingan and M.~Lemos,
On the circuit-cocircuit intersection conjecture,
\emph{Graphs Combin.} \textbf{22} (2006), no.~4, 471--480.

\bibitem{Oxley1984}
J.~G.~Oxley,
On the intersections of circuits and cocircuits in matroids,
\emph{Combinatorica} \textbf{4} (1984), no.~2--3, 187--195.

\bibitem{Oxley2011}
J.~Oxley,
\emph{Matroid Theory}, second ed.,
Oxford Graduate Texts in Mathematics, vol.~21,
Oxford University Press, Oxford, 2011.

\bibitem{Shin2021}
J.~Shin,
The circuit-cocircuit intersection conjecture for intersection size $k\leq6$,
preprint, 2021, arXiv:2111.02605.

\bibitem{Shin2023}
J.~Shin,
The circuit-cocircuit intersection conjecture for intersection size $\leq7$,
preprint, 2023, arXiv:2308.02859.
\end{thebibliography}
\end{document}